\documentclass[12pt]{article}
\usepackage{ulem}
\usepackage{graphicx,amsmath,amsfonts,amssymb,color,mathrsfs, amsmath,amsthm}
\usepackage{epsfig}
\usepackage{CJK}
\usepackage{graphicx}
\usepackage{subfigure}

\newcommand{\chuhao}{\fontsize{19pt}{\baselineskip}\selectfont}

  \newtheorem{theorem}{Theorem}[section]
 \newtheorem{lemma}{Lemma}[section]

 \newtheorem{remark}{Remark}[section]

 \newtheorem{example}{Example}[section]
 
 \newtheorem{corollary}{Corollary}[section]

\makeatletter

\newcommand{\Rmnum}[1]{\expandafter\@slowromancap\romannumeral #1@}
\makeatother
\usepackage{epstopdf}

\usepackage{pifont}

\title{\bf\color{black} \chuhao{Potentials of skip-free Markov chains}}

\author{
Wendi Li$^a$,\ \ Jinpeng Liu$^a$,\ \  Yuanyuan Liu$^{b}$\\
\small{\textit{ $^a$School of Mathematics and Statistics, Xidian University,}}\\
\small{\textit{Xi'an, 710071, Shaanxi, P.R. China}}\\
\small{\textit{ $^b$ School of Mathematics and Statistics, HNP-LAMA, Central South University,},}\\
\small{\textit{Changsha, 410083, Hunan, P.R. China}}\\
}
\date{\today}

\begin{document}

 \maketitle

\begin{abstract}
Potential theory has important applications in various fields such as physics, finance, and biology.
In this paper, we investigate the potentials of two classic types of discrete-time skip-free Markov chains: upward skip-free and downward skip-free Markov chains.
The key to deriving these potentials lies in the use of truncation approximation techniques.
The results are then applied to \( GI/M/1 \) queues and \( M/G/1 \) queues, and further extended to continuous-time skip-free Markov chains.
\vskip 0.2cm
\noindent {\bf MSC Classiﬁcation:} \ \ 60J10, 60J45.
\end{abstract}

\section{Introduction}\label{sec1}

Classical potential theory is the study of the electric potential $\boldsymbol{\phi}$ that arise from electric fields $\boldsymbol{E}$, governed by relationships such as
\[
-\nabla\boldsymbol{\phi} = \boldsymbol{E},
\]
where $\nabla$ is the gradient operator.
In addition, several other physical theories, such as Newton's theory of gravity, fluid flow, and the diffusion of heat, can also be described within the framework of potential theory; see, e.g., \cite{N97}.
Choquet and Deny \cite{CD56} demonstrated that the basic principles of potential theory apply to potentials $\boldsymbol{\phi}$ defined by more general operators $\nabla$.
Subsequently, Doob \cite{D59} extended classical potential theory to the study of Brownian motion.
Meanwhile, Hunt \cite{H57},  Kemeny and Snell \cite{KS61} explored potential theory of more general Markov chains.

Speciﬁcally, let $\boldsymbol{X}=\{X_k, k\geq0\}$ be an irreducible, discrete-time Markov chain (DTMC) on the state space ${\mathbb{S}}=\{0,1,2,\ldots\}$.
Assume that $\boldsymbol{X}$ is time-homogeneous,
and let $\boldsymbol{P}=(P(i,j))_{i,j\in {\mathbb{S}}}$ be the one step transition probability matrix of $\boldsymbol{X}$ with $\sum_{j\in {\mathbb{S}}}P(i,j)=1$ for any $i\in {\mathbb{S}}$.
Furthermore, assume that $\boldsymbol{c}$ is a finite nonnegative function (or vector) on ${\mathbb{S}}$.
We say that the function $\boldsymbol{\phi}$ is the potential with respect to the chain $\boldsymbol{X}$ and the function $\boldsymbol{c}$ if \(\boldsymbol{\phi}\) satisfies the following equation:
\begin{equation}\label{equation-1}
(\boldsymbol{I}-\boldsymbol{P})\boldsymbol{\phi}=\boldsymbol{c},
\end{equation}
where $\boldsymbol{I}$ is the identity matrix.
Eq.(\ref{equation-1}) is also known as Poisson's equation, in which case $\boldsymbol{c}$ and $\boldsymbol{\phi}$ are referred to as the forcing function and the solution function of Poisson's equation, respectively.

Now, let $\mathbb{E}_{i}[\cdot]:=\mathbb{E}[\cdot|X_0=i]$ denote the conditional expectation given the initial state $i\in{\mathbb{S}}$.
It is established that
\begin{equation*}
\phi(i)=\mathbb{E}_{i}\bigg[\sum_{k=0}^{\infty}c(X_k)\bigg], \ \ i\in{\mathbb{S}},
\end{equation*}
is the minimal nonnegative solution of Eq.(\ref{equation-1}),
in the sense that if $\boldsymbol{\phi}'$ is any nonnegative solution of Eq.(\ref{equation-1}), then $\phi(i)\leq \phi'(i)$ for all $i\in{\mathbb{S}}$; see, e.g., \cite{HG88,N97}.
The potential $\boldsymbol{\phi}$ is closely related to the Green matrix $\boldsymbol{G} = (G(i,j))_{i,j \in \mathbb{S}}$, which is defined as
\[
\boldsymbol{G} := \sum_{k=0}^{\infty} \boldsymbol{P}^k.
\]
It follows from Fubini's theorem that
\begin{equation}\label{Phi-and-G}
\phi(i) = \sum_{k=0}^{\infty} \mathbb{E}_i[c(X_k)] = \sum_{k=0}^{\infty} (\boldsymbol{P}^n\boldsymbol{c})(i)=(\boldsymbol{G}\boldsymbol{c})(i), \ \  i\in{\mathbb{S}}.
\end{equation}

In finance, the function $\boldsymbol{c}$ is usually regarded as a cost (or reward) function, and the potential $\boldsymbol\phi$ has the interpretation of the expected total cost (or reward); see, \cite{B07,SYL17}.
In biology, it has been linked to total food consumption and toxin production by bacteria \cite{P67}, as well as to the overall cost of epidemics \cite{J72}.
The exact expressions for the potential $\boldsymbol{\phi}$ and the Green matrix $\boldsymbol{G}$ have been investigated for several Markov chains.
For example, Ramakrishnan \cite{R84} first investigated the potential of finitely additive Markov chains.
Hernández-Suárez and Castillo-Chavez \cite{HC99} then derived an expression for the potential of birth-death processes with absorbing states, under the condition that $c(i)=i$.
Later, Stefanov and Wang \cite{SS00} extended the result from \cite{HC99} to cases where $\boldsymbol{c}$ is eventually non-decreasing.
In subsequent work, Pollett \cite{P03} removed the assumption of absorbing states in birth-death processes from \cite{HC99, P03},
and also studied the potential of birth, death, and catastrophe processes with absorbing states.
For more recent developments in this area, please refer to \cite{SM17, CP19, DM21} and their respective references.

In this paper, we will focus on the potential $\boldsymbol\phi$ of two classic types of skip-free Markov chains, including upward skip-free Markov chains and downward skip-free Markov chains.
We call $\boldsymbol{X}$ an upward skip-free Markov chain if its transition probability matrix $\boldsymbol{P}$ satisfies $P(i,i+1)>0$ for all $i\geq0$, and $P(i,k)=0$ for all $k-i\geq2$, i.e.,
\begin{equation*}
 \boldsymbol{P}=\left (
 \begin{array}{ccccccc}
   P(0,0) & P(0,1) & 0 & 0&\cdots \\
   P(1,0) & P(1,1) & P(1,2) & 0 & \cdots \\
   P(2,0) & P(2,1) & P(2,2) &P(2,3) & \cdots\\
   P(3,0) & P(3,1) & P(3,2) &P(3,3) &\cdots  \\
   \vdots &\vdots &\vdots &\vdots &\ddots \\
   \end{array}\right ).
\end{equation*}
As the dual process of the upward skip-free Markov chains, the transition probability matrix $\boldsymbol{P}$ of the downward skip-free Markov chain satisfies $P(i,i-1)>0$ for all $i\geq1$,
and $P(i,k)=0$ for all $i-k\geq2$, i.e.,
\begin{equation*}
\boldsymbol{P}=\left (
 \begin{array}{ccccccc}
   P(0,0) &P(0,1) &P(0,2) &P(0,3)&\cdots \\
   P(1,0) & P(1,1) & P(1,2) &P(1,3) & \cdots \\
 0 & P(2,1) & P(2,2) &P(2,3) & \cdots\\
 0 & 0& P(3,2) &P(3,3) &\cdots  \\
   \vdots &\vdots &\vdots &\vdots &\ddots \\
   \end{array}\right ).
\end{equation*}

Skip-free Markov chains have attracted lots of attention due to their importance in practical applications, such as queueing theory, economic theory, population models,
chemical models, biology theory and other disciplines; see, e.g., \cite{A86,A03,BGR82} and the references therein.
On the theoretical side,  Abate and Whitt \cite{AW89} made a spectral analysis of skip-free Markov chain and gave the Laplace transform of the upward first hitting time.
Zhang et al. \cite{JPYP2015,MYZ22}  investigated the criteria for recurrence and  ergodicity of skip-free Markov chains.
Liu et al. \cite{JLY14,LLZ22} developed  Poisson's equation and central limit theorems for positive recurrent skip-free Markov chains.
In addition, there are some other studies on skip-free Markov chains,
such as the separation cutoff, maximal inequalities, scale functions and fluctuation theory, see \cite{MZZ16,J19,AV19,LPW24} and the references therein.

In previous work \cite{P03} , Pollet primarily utilized the minimal non-negativity of $\boldsymbol{\phi}$,
combined with the tridiagonal structure of the transition probability function of birth-death processes, and derived the expression for the potential of birth-death processes through iterative calculations.
However, when the transition structure of Markov chains becomes highly complex, Pollett's method approach is no longer applicable.
Consequently, to obtain explicit expressions for the potentials of skip-free Markov chains,
we will develop truncation approximation techniques for potential theory of general Markov chains.
These truncation techniques serve two main purposes: first, they are not constrained by the transition structure of the Markov chains when solving for the potential;
second, they provide a theoretical foundation for the numerical approximation of potentials.

The rest of this paper is organized into three sections:
Section \ref{sec2} presents the main results, including the truncation approximation theorems for the potential and the Green matrix,
as well as explicit expressions for the potentials of both upward and downward skip-free Markov chains.
Section \ref{sec3} illustrates these results with two examples: the $GI/M/1$ queue and the $M/G/1$ queue.
Finally, Section \ref{sec4} extends the main results from discrete-time skip-free Markov chains to continuous-time skip-free Markov chains.

\section{Main results}\label{sec2}

\subsection{General Markov chains}\label{subsec1}

We begin by introducing the truncation approximation method in the context of general DTMCs.
For an irreducible transition probability matrix $\boldsymbol{P}$, we first truncate the $(n+1)\times(n+1)$ northwest corner of $\boldsymbol{P}$ and denote it as $_{(n)}\boldsymbol{P}$, specifically,
\begin{equation*}
 {_{(n)}\boldsymbol{P}}=\left (
 \begin{array}{ccccccc}
   P(0,0) & P(0,1) & P(0,2) &\cdots  & P(0,n)\\
   P(1,0) & P(1,1) & P(1,2) & \cdots & P(1,n)  \\
   P(2,0) & P(2,1) & P(2,2) & \cdots & P(2,n) \\
     \vdots &\vdots  &\vdots &\ddots &\vdots\\
   P(n,0) & P(n,1)& P(n,2)  &\cdots & P(n,n)  \\
   \end{array}\right ).
\end{equation*}

Then, let ${_{(n)}\boldsymbol{X}}=\{{_{(n)}{X}_k}, k\geq 0\}$ be the DTMC on the state space $_{(n)}\mathbb{S}=\{0, 1, 2, \ldots, n\}$ with transition probability matrix $_{(n)}\boldsymbol{P}$,
and let ${_{(n)}\boldsymbol{c}}$ represent the truncation function, consisting of the first $n+1$ elements of $\boldsymbol{c}$.
The potential $_{(n)}\boldsymbol{\phi}$ with respect to the chain ${_{(n)}\boldsymbol{X}}$ and the function ${_{(n)}\boldsymbol{c}}$ is expressed as
\begin{equation*}
_{(n)}\phi(i)=\mathbb{E}_{i}\bigg[\sum_{k=0}^{\infty}{_{(n)}c}\big({_{(n)}X_k}\big)\bigg], \ \ i\in {_{(n)}\mathbb{S}}.
\end{equation*}
Since $_{(n)}\boldsymbol{P}$ is irreducible and non-stochastic, its Green matrix
\[
_{(n)}\boldsymbol{G} = \sum_{k=0}^{\infty} {_{(n)}\boldsymbol{P}}^k = \left({_{(n)}\boldsymbol{I}} - {_{(n)}\boldsymbol{P}}\right)^{-1}
\]
exists and is finite, where ${_{(n)}\boldsymbol{I}}$ is the $(n+1)$-dimensional identity matrix.
Therefore, according to Eq.(\ref{Phi-and-G}), the potential $_{(n)}\boldsymbol{\phi}$ also exists and is finite.
The following result describes the relationship between $_{(n)}\boldsymbol{\phi}$  and $\boldsymbol{\phi}$.

\begin{theorem}\label{theorem-general-mc}
Assume that the DTMC $\boldsymbol{X}$ is irreducible and the nonnegative function $\boldsymbol{c}$ on $\mathbb{S}$ is finite. Then, we have
\[\lim_{n\rightarrow\infty}{_{(n)}}\phi(i)\uparrow\phi(i),\ \ i\in\mathbb{S}.\]
\end{theorem}
\begin{proof}
It follows from Eq.(\ref{equation-1}) that $_{(n)}\boldsymbol{\phi}$ is the minimal nonnegative solution (also the unique solution) of the following equation:
\begin{equation}\label{proposition-1-3}
\left({_{(n)}\boldsymbol{I}} - {_{(n)}\boldsymbol{P}}\right){_{(n)}\boldsymbol{\phi}}={_{(n)}\boldsymbol{c}}.
\end{equation}

We now extend the potential \({_{(n)}\boldsymbol{\phi}}\) and the function \({_{(n)}\boldsymbol{c}}\) from \({_{(n)}\mathbb{S}}\) to \(\mathbb{S}\).
The extended potential ${_{(n)}}\boldsymbol{\check{\phi}}$ is defined as
\begin{equation}\label{n-phi}
{_{(n)}}\boldsymbol{\check{\phi}} := ({_{(n)}}\phi(0), {_{(n)}}\phi(1), \cdots, {_{(n)}}\phi(n), 0, 0, \cdots),
\end{equation}
and the extended function ${_{(n)}}\boldsymbol{\check{c}}$ is represented as
\begin{equation}\label{n-c}
{_{(n)}}\boldsymbol{\check{c}} := (c(0), c(1), \cdots, c(n), 0, 0, \cdots).
\end{equation}

Similarly, we define the extended transition probability matrix \({_{(n)}\boldsymbol{\check{P}}}\) as follows
\begin{equation}\label{n-P}
{_{(n)}\boldsymbol{\check{P}}} := \begin{pmatrix}
   P(0,0) & P(0,1) & \cdots &P(0,n) & 0 & 0 & \cdots \\
   P(1,0) & P(1,1) & \cdots &P(1,n) & 0 & 0 & \cdots \\
   \vdots & \vdots & \ddots & \vdots & \vdots & \vdots & \cdots \\
   P(n,0) & P(n,1) & \cdots & P(n,n) & 0 & 0 & \cdots \\
   0 & 0 & \cdots & 0 & 0 & 0 & \cdots \\
   0 & 0 & \cdots & 0 & 0 & 0 & \cdots \\
   \vdots & \vdots & \vdots & \vdots & \vdots & \vdots & \ddots \\
\end{pmatrix}
\end{equation}
According to Eqs.(\ref{n-phi})--(\ref{n-P}), we rewrite Eq.(\ref{proposition-1-3}) as
\begin{equation}\label{proposition-1-4}
{_{(n)}}\boldsymbol{\check{\phi}}={_{(n)}\boldsymbol{\check{P}}}{{_{(n)}}\boldsymbol{\check{\phi}}}+{_{(n)}}\boldsymbol{\check{c}}
\end{equation}
Furthermore, we rewrite Eq.(\ref{equation-1}) as
\begin{equation}\label{proposition-1-5}
\boldsymbol{\phi}=\boldsymbol{P}\boldsymbol{\phi}+\boldsymbol{{c}}.
\end{equation}

It is evident that
\begin{equation}\label{proposition-1-6}
\lim_{n\rightarrow\infty} {_{(n)}\boldsymbol{\check{P}}}\uparrow\boldsymbol{P} \quad \text{and} \quad \lim_{n\rightarrow\infty} {_{(n)}}\boldsymbol{\check{c}} \uparrow \boldsymbol{c}.
\end{equation}
Therefore, by applying Eqs.(\ref{proposition-1-4})--(\ref{proposition-1-6}) and Theorem 2.7 from \cite{c04}, we conclude that
\begin{equation*}\label{proposition-1-7}
\lim_{n\rightarrow\infty} {{_{(n)}}\boldsymbol{\check{\phi}}} \uparrow \boldsymbol{\phi},
\end{equation*}
from which that
\begin{equation*}\label{proposition-1-7}
\lim_{n\rightarrow\infty} {{_{(n)}}{{\phi}(i)}} \uparrow {\phi(i)}, \ \ i\in{\mathbb{S}}.
\end{equation*}
This completes the proof of the theorem.
\end{proof}

Using a similar technique as in Theorem \ref{theorem-general-mc}, we can also establish the convergence theorem for the Green matrices \({_{(n)}\boldsymbol{G}}\) and \(\boldsymbol{G}\).

\begin{corollary}\label{col-G}
Assume that the DTMC $\boldsymbol{X}$ is irreducible. Then, we have
\[\lim_{n\rightarrow\infty}{_{(n)}G(i,j)}\uparrow G(i,j), \ \ i,j\in E. \]
\end{corollary}
\begin{proof}
According to Eqs.(\ref{equation-1}) and (\ref{Phi-and-G}), we know that the Green matrix \(\boldsymbol{G}\) is the minimal nonnegative solution of following equation:
\begin{equation}\label{min-nonnegative-solution-of-G}
(\boldsymbol{I}-\boldsymbol{P})\boldsymbol{G}=\boldsymbol{I}.
\end{equation}

For any fixed state $i\in\mathbb{S}$, let $\boldsymbol{g_i}$ be the $(i+1){th}$-column of  the Green matrix \(\boldsymbol{G}\), i.e.,
\[\boldsymbol{g_i}=\left({{G(0,i)}},{{G(1,i)}},{{G(2,i)}},\cdots\right).\]
And let ${\boldsymbol{e_i}}$ be the function with one in the $(i+1){th}$ position, zeros elsewhere.
It follows from Eq.(\ref{min-nonnegative-solution-of-G}) that $\boldsymbol{g_i}$ is the minimal nonnegative solution of the following equation:
\[{\boldsymbol{g_i}}=\boldsymbol{P}{\boldsymbol{g_i}}+{\boldsymbol{e_i}}, \ \ i\in \mathbb{S}.\]

Similarly, let $_{(n)}\boldsymbol{g_i}$ be the $(i+1){th}$-column of the Green matrix \({_{(n)}\boldsymbol{G}}\) for $i\in {_{(n)}}\mathbb{S}$, i.e.,
\[_{(n)}\boldsymbol{g_i}=\left({_{(n)}{G(0,i)}},{_{(n)}{G(1,i)}},\cdots,{_{(n)}{G(n,i)}}\right).\]
Then,  $_{(n)}\boldsymbol{g_i}$ is the  minimal nonnegative solution of the following equation:
\[{_{(n)}\boldsymbol{g_i}}={_{(n)}\boldsymbol{P}}{_{(n)}\boldsymbol{g_i}}+{_{(n)}\boldsymbol{e_i}}, \ \ i\in {_{(n)}}\mathbb{S},\]
where ${_{(n)}\boldsymbol{e}_{i}}$ is the $(n+1)$-function with one in the $(i+1)th$ position, zeros elsewhere.

The rest of the proof is similar to Theorem \ref{theorem-general-mc}, which is omitted here.
\end{proof}

\subsection{Upward skip-free Markov chains}\label{subsec2}

To determine the potential $\boldsymbol{\phi}$ of the upward skip-free Markov chains, we need to introduce the following notations and lemma (see, e.g., \cite{JPYP2015}):
\begin{equation*}
{P_{n}^{(k-)}}=\sum_{i=0}^{k}P(n,i),\ \ 0\leq k<n,
\end{equation*}
and
\begin{equation*}
F_{i}^{(i)}=1, \ \ i\geq 0,\ \ F_{n}^{(i)}=\frac{1}{P(n,n+1)}\sum_{k=i}^{n-1}{P_{n}^{(k-)}}F_{k}^{(i)}, \ \ 0\leq i<n.
\end{equation*}

\begin{lemma}\label{lemma-upward}
Given a finite function $\boldsymbol{c}$ on $\mathbb{S}$, the sequence of functions $\{f(n), n \in \mathbb{N}\}$ is defined by
\[f(i)=\frac{c(i)}{P(i,i+1)}, \ \ f(n)=\frac{1}{P(n,n+1)}\left(\sum_{k=i}^{n-1}{P_{n}^{(k-)}}f(k)+c(n)\right),\ \ n\geq i+1,\]
has the following equivalent presentation:
\[f(n) =\sum_{k=i}^{n}\frac{F_{n}^{(k)}c(k)}{P(k,k+1)},\ \  n\geq i.\]
\end{lemma}

\begin{theorem}\label{theorem-1}
Assume that $\boldsymbol{X}$ is an irreducible upward skip-free Markov chain and the nonnegative function $\boldsymbol{c}$ on $\mathbb{S}$ is finite.
Then, we have
 \begin{equation}\label{phi-single-birth}
  \phi(i)=\sum_{m=i}^{\infty}\sum_{k=0}^{m}\frac{F_{m}^{(k)}c(k)}{P(k,k+1)}, \ \ i \in\mathbb{S},
 \end{equation}
and this is finite if and only if $\phi(0)<\infty$.
\end{theorem}
\begin{proof}
We first calculate the potential $_{(n)}\boldsymbol{\phi}$ of the upward skip-free Markov chain $_{(n)}\boldsymbol{X}$
with transition probability matrix $_{(n)}\boldsymbol{P}$, which is given by
\begin{equation*}
 {_{(n)}\boldsymbol{P}}=\left (
 \begin{array}{ccccccc}
   P(0,0) & P(0,1) & 0&\cdots  & 0\\
   P(1,0) & P(1,1) & P(1,2) & \cdots &0 \\
   \vdots &\vdots  &\vdots &\ddots &\vdots\\
   P(n-1,0) & P(n-1,1) & P(n-1,2) & \cdots & P(n-1,n) \\
   P(n,0) & P(n,1)& P(n,2)  &\cdots & P(n,n)  \\
   \end{array}\right ).
\end{equation*}

According to Eq.(\ref{proposition-1-3}) and the structure of $_{(n)}\boldsymbol{P}$, we have
\begin{equation*}\label{single-birth-1}
{_{(n)}}\phi(i)-\sum_{j=0}^{i+1} P(i,j) {_{(n)}}\phi(j)=c(i), \ \  0\leq i\leq n-1,
\end{equation*}
and
\begin{equation*}\label{single-birth-1}
{_{(n)}}\phi(n)-\sum_{j=0}^{n} P(n,j) {_{(n)}}\phi(j)=c(n).
\end{equation*}

Now, we set ${_{(n)}}\phi_{n,0}(n)=0$, and let $\Delta(i)={_{(n)}}\phi(i)-{_{(n)}}\phi(i+1)$, $0\leq i \leq n$.
Hence, we rewrite the above equations as
\begin{equation*}
P(0,1)\Delta(0)=c(0)
\end{equation*}
and
\begin{equation*}
P(i,i+1)\Delta(i)-\sum_{k=0}^{i-1}{P}_{i}^{(k-)}\Delta(k) = c(i),\  \  1\leq i\leq n.
\end{equation*}
By simple calculation, we obtain
\begin{equation}\label{theorem-1-1}
\Delta(0)=\frac{c(0)}{P(0,1)},\ \ \Delta(i)=\frac{1}{P(i,i+1)}\left(c(i)+\sum_{k=0}^{i-1}{P}_{i}^{(k-)}\Delta_k \right) ,\  \  1\leq i\leq n.
\end{equation}
It follows from Eq.(\ref{theorem-1-1})  and Lemma \ref{lemma-upward} that
\begin{equation}\label{theorem-1-3}
  {_{(n)}}\phi(i)-{_{(n)}}\phi(i+1)=\Delta(i)=\sum_{k=0}^{i}\frac{F_{i}^{(k)}c(k)}{P(k,k+1)}, \ \ 0\leq i \leq n.
\end{equation}
Summing Eq.(\ref{theorem-1-3}) from $i$ to $n$  gives
\[{_{(n)}}\phi(i)=\sum_{m=i}^{n}\sum_{k=0}^{m}\frac{F_{m}^{(k)}c(k)}{P(k,k+1)}, \ \ 0\leq i \leq n.\]
From Theorem \ref{theorem-general-mc}, we have
\[\lim_{n\rightarrow\infty}{_{(n)}}\phi(i)=\lim_{n\rightarrow\infty}\sum_{m=i}^{n}\sum_{k=0}^{m}\frac{F_{m}^{(k)}c(k)}{P(k,k+1)}
=\sum_{m=i}^{\infty}\sum_{k=0}^{m}\frac{F_{n}^{(k)}c(k)}{P(k,k+1)}=\phi(i).\]

For a finite non-negative solution, we require $\phi(i)< \infty$ for each $i\in\mathbb{S}$.
From Eq.(\ref{phi-single-birth}), the potential $\boldsymbol\phi$ is a monotone non-decreasing function.
Thus, $\boldsymbol\phi$ is finite  if and only if $\phi(0)<\infty$.
We have proved the result.
\end{proof}

By Corollary \ref{col-G} and Theorem \ref{theorem-1}, we obtain the following result about the Green matrix of  upward skip-free Markov chains.

\begin{corollary}\label{G-upward-process}
Assume that $\boldsymbol{X}$ is an irreducible upward skip-free Markov chain.
Then, its Green matrix $\boldsymbol{G}$ is given by
\[G(i,j)=\frac{1}{P(j,j+1)}\sum_{m=\max\{i,j\}}^{\infty}F_m^{(j)}, \ \ i,j\in\mathbb{S}.\]
\end{corollary}

\begin{remark} \label{remark-transient-upward}
From Corollary \ref{G-upward-process}, we know that the upward skip-free Markov chain  is transient if and only if
\[\sum_{n=0}^{\infty}F_{n}^{(0)}<\infty.\]
More results on recurrence of upward skip-free processes can be found in  \cite{JPYP2015}.
\end{remark}

\subsection{Downward skip-free Markov chains}\label{subsec3}

Similar to handling upward skip-free Markov chains, we introduce the follow notation and the lemma
to better derive the potential $\boldsymbol{\phi}$ of downward skip-free Markov chains (see, e.g., \cite{MYZ22}):
\[
P_{m}^{(k+)}=\sum_{i=k}^{\infty}P(m,i),\ \ k>m\geq0,
\]
and
\[
H_{i}^{(i)}=1, \ \ H_{m}^{(i)}=\frac{1}{P(m,m-1)}\sum_{k=m+1}^{i}P_{m}^{(k+)}H_{k}^{(i)}, \ 1\leq m<i.
\]
\begin{lemma}\label{lemma-single-death}
Given a finite function $\boldsymbol{c}$ on $\mathbb{S}$, the sequence of functions $\{h(n), n \in \mathbb{N}\}$ is defined by
\[h(n+1)=0, \ \ h(i)=\frac{1}{P(i,i-1)}\left(\sum_{k=i+1}^{n}P_{i}^{(k+)}h(k)+c(i)\right),\ \ i\leq n,\]
has the following equivalent presentation:
\[h(i) =\sum_{k=i}^{n}\frac{H_{i}^{(k)}c(k)}{P(k,k-1)},\ \  i\leq n.\]
\end{lemma}

\begin{theorem}\label{theorem-2}
Assume that $\boldsymbol{X}$ is an irreducible downward skip-free Markov chain and the nonnegative function $\boldsymbol{c}$ on $\mathbb{S}$ is finite.
Then, we have
 \begin{equation}\label{phi-single-death}
  \phi(i)=\delta M(i)-\sum_{m=i+1}^{\infty}\sum_{k=m}^{\infty}\frac{H_m^{(k)}c(k)}{P(k,k-1)}, \ \ i \in\mathbb{S},
 \end{equation}
 where
 \[\delta=\sum_{i=1}^{\infty} P_{0}^{(i+)}\sum_{k=i}^{\infty}\frac{H_i^{(k)}c(k)}{P(k,k-1)}+c(0),\ \ M(i)=\lim_{n\rightarrow\infty}\frac{\sum_{m=i+1}^{n}H_m^{(n)}}{\sum_{i=1}^{n} P_{0}^{(i+)}H_i^{(n)}},\]
and this is finite if and only if $\phi(0)<\infty$.
\end{theorem}
\begin{proof}
We first truncate the downward skip-free transition probability matrix of the downward skip-free Markov chain,  and $_{(n)}\boldsymbol{P}$ is given by
\begin{equation*}
 {_{(n)}\boldsymbol{P}}=\left (
 \begin{array}{ccccccc}
   P(0,0) & P(0,1) &\cdots & P(0,n-1) & P(0,n)\\
   P(1,0) & P(1,1) & \cdots & P(1,n-1) &P(1,n) \\
   \vdots &\vdots &\ddots  &\vdots &\vdots\\
   0 & 0& \cdots  & P(n-1,n-1) & P(n-1,n) \\
  0& 0&  \cdots& P(n,n-1)  & P(n,n)  \\
   \end{array}\right ).
\end{equation*}

It follows from $_{(n)}\boldsymbol{P}$ and Eq.(\ref{proposition-1-3}) that
\begin{equation}\label{single-death-1}
{_{(n)}}\phi(0)-\sum_{j=0}^{n} P(0,j) {_{(n)}}\phi(j)=c(0)
\end{equation}
and
\begin{equation}\label{single-death-2}
{_{(n)}}\phi(i)-\sum_{j=i-1}^{n} P(i,j) {_{(n)}}\phi(j)=c(i), \ \  1\leq i\leq n.
\end{equation}
Now, let  $\Theta(i)={_{(n)}}\phi(i)-{_{(n)}}\phi(i-1)$, $1\leq i \leq n$.
Eqs.(\ref{single-death-1}) and (\ref{single-death-2}) can be written as
\begin{equation}\label{theorem-2-1}
 \sum_{i=1}^{{n}} {P_{0}^{(i+)}}\Theta(i)-{P_{0}^{((n+1)+)}}{_{(n)}}\phi(n)+c(0)=0,
\end{equation}
and
\begin{equation}\label{theorem-2-2}
P(i,i-1)\Theta(i)=\sum_{k=i+1}^{{n}}P_{i}^{(k+)}\Theta(k)-{P_{i}^{((n+1)+)}}{_{(n)}}\phi(n)+c(i),\  \  1\leq i\leq n.
\end{equation}
According to Eq.(\ref{theorem-2-2}), we have
\begin{equation}\label{theorem-2-3}
\Theta(i)=\frac{1}{P(i,i-1)}\left(\sum_{k=i+1}^{n}P_{i}^{(k+)}\Theta(k) + c(i)-{P_{i}^{((n+1)+)}}{_{(n)}}\phi(n)\right),\  \  1\leq i\leq n.
\end{equation}
By Lemma \ref{lemma-single-death} and Eq.(\ref{theorem-2-3}), we obtain that
\begin{eqnarray}
\nonumber  \Theta(i) &=& \sum_{k=i}^{{n}}\frac{H_i^{(k)}\left(c({k})-{P_{k}^{((n+1)+)}}{_{(n)}}\phi(n)\right)}{P({k,k-1})} \\
 \label{theorem-2-5}  &=& \sum_{k=i}^{{n}}\frac{H_i^{(k)}c({k})}{P({k,k-1})} -\sum_{k=i}^{{n}}\frac{H_i^{(k)}{P_{k}^{((n+1)+)}}}{P({k,k-1})}{_{(n)}}\phi(n),\ \ 1\leq i\leq n.
\end{eqnarray}
Let $m(i)=\sum_{k=i}^{{n}}\frac{H_i^{(k)}{P_{k}^{((n+1)+)}}}{P({k,k-1})}$, $1\leq i\leq{n}$, and it follows from Lemma \ref{lemma-single-death} that
\begin{equation}\label{m(i)-1}
m(i)=\frac{1}{P(i,i-1)}\left(\sum_{k=i+1}^{n}P_{i}^{(k+)}m(k)+P_{i}^{((n+1)+)}\right),\ \ 1\leq i\leq{n}.
\end{equation}
According to the definition of $H_{i}^{(n)}$, we have
\begin{eqnarray}
\nonumber  H_{i}^{(n+1)} &=& \frac{1}{P(i,i-1)}\left(\sum_{k=i+1}^{n+1}P_{i}^{(k+)}H_{k}^{(n+1)}\right) \\
 \nonumber  &=& \frac{1}{P(i,i-1)}\left(\sum_{k=i+1}^{n}P_{i}^{(k+)}H_{k}^{(n+1)}+P_{i}^{((n+1)+)}H_{n+1}^{(n+1)}\right)\\
 \label{m(i)-2}  &=& \frac{1}{P(i,i-1)}\left(\sum_{k=i+1}^{n}P_{i}^{(k+)}H_{k}^{(n+1)}+P_{i}^{((n+1)+)}\right), \ \ 1\leq i\leq{n}.
\end{eqnarray}
Combining Eqs.(\ref{m(i)-1}) and (\ref{m(i)-2}), we obtain $m(i) = H_{i}^{(n+1)}$ for $1 \leq i \leq n$, and thus Eq.(\ref{theorem-2-5}) becomes
\begin{equation}\label{theorem-2-6-1}
\Theta(i)=\sum_{k=i}^{{n}}\frac{H_i^{(k)}c({k})}{P({k,k-1})} -H_{i}^{(n+1)}{_{(n)}}\phi(n), 1\leq i\leq n.
\end{equation}
By substituting Eq.(\ref{theorem-2-6-1}) into Eq.(\ref{theorem-2-1}), we have
\begin{equation*}
 \sum_{i=1}^{n} P_{0}^{(i+)}\sum_{k=i}^{n}\frac{H_i^{(k)}c(i)}{P(k,k-1)}-\sum_{i=1}^{{n+1}} P_{0}^{(i+)}H_i^{(n+1)}{_{(n)}}\phi(n)+c(0)=0.
\end{equation*}
Clearly,
\begin{equation}\label{theorem-2-6}
{{_{(n)}}\phi(n)}=\frac{\sum_{i=1}^{n} P_{0}^{(i+)}\sum_{k=i}^{n}\frac{H_i^{(k)}c(i)}{P(k,k-1)}+c(0)}{\sum_{i=1}^{{n+1}} P_{0}^{(i+)}G_i^{(n+1)}} .
\end{equation}
It follows from Eqs.(\ref{theorem-2-5}) and (\ref{theorem-2-6}) that
\begin{eqnarray*}\label{theorem-2-7}
  {_{(n)}}\phi(i) &=& \sum_{m=i+1}^{n+1}H_m^{(n+1)}\frac{\sum_{i=1}^{n} P_{0}^{(i+)}\sum_{k=i}^{n}\frac{H_i^{(k)}c(k)}{P(k,k-1)}+c(0)}{\sum_{i=1}^{{n+1}} P_{0}^{(i+)}H_i^{(n+1)}}-\sum_{m=i+1}^{n}\sum_{k=m}^{n}\frac{H_m^{(k)}c(k)}{P(k,k-1)} \\
 &=& \frac{\sum_{m=i+1}^{{n+1}}H_m^{(n+1)}}{\sum_{i=1}^{{n+1}} P_{0}^{(i+)}H_i^{(n+1)}}\left({\sum_{i=1}^{n} P_{0}^{(i+)}\sum_{k=i}^{n}\frac{H_i^{(k)}c(k)}{P(k,k-1)}+c(0)}\right)
-\sum_{m=i+1}^{n}\sum_{k=m}^{n}\frac{H_m^{(k)}c(k)}{P(k,k-1)}.
\end{eqnarray*}
Moreover, by applying Theorem \ref{theorem-general-mc} along with the equation above, we derive the potential $\boldsymbol{\phi}$ for downward skip-free Markov chains, as shown in Eq.(\ref{phi-single-death}).

Finally, we demonstrate that $\boldsymbol{\phi}$ is finite if and only if $\phi(0) < \infty$. The sufficiency is evident, so we focus on proving the necessity.
Assuming $\phi(0) < \infty$, we find that
$$\sum_{m=1}^{\infty}\sum_{k=m}^{\infty}\frac{H_m^{(k)}c(k)}{P(k,k-1)} < \infty.$$
Since $P_{0}^{(i+)} < 1$ for $i \geq 1$ and $c(0) < \infty$, it follows that
\begin{equation}\label{finite-1}
\delta=\sum_{i=1}^{\infty} P_{0}^{(i+)}\sum_{k=i}^{\infty}\frac{H_i^{(k)}c(k)}{P(k,k-1)}+c(0)<\sum_{m=1}^{\infty}\sum_{k=m}^{\infty}\frac{H_m^{(k)}c(k)}{P(k,k-1)}+c(0)<\infty.
\end{equation}
Furthermore, we have
$$M(0) = \lim_{n\rightarrow\infty}\frac{\sum_{m=1}^{n}H_m^{(n)}}{\sum_{i=1}^{n} P_{0}^{(i+)}H_i^{(n)}} < \infty.$$
Since $H_{m}^{(n)}$ for $1 \leq m < n$ is nonnegative and finite, it holds that
$$\frac{\sum_{m=i+1}^{{n}}H_m^{(n)}}{\sum_{i=1}^{{n}} P_{0}^{(i+)}H_i^{(n)}}\leq\frac{\sum_{m=1}^{n}H_m^{(n)}}{\sum_{i=1}^{n} P_{0}^{(i+)}H_i^{(n)}}, \ \ 1\leq i\leq n,$$
from which that
\begin{equation}\label{finite-2}
M(i)=\lim_{n\rightarrow\infty}\frac{\sum_{m=i+1}^{{n}}H_m^{(n)}}{\sum_{i=1}^{{n}} P_{0}^{(i+)}H_i^{(n)}}
\leq\lim_{n\rightarrow\infty}\frac{\sum_{m=1}^{n}H_m^{(n)}}{\sum_{i=1}^{n} P_{0}^{(i+)}H_i^{(n)}}=M(0)<\infty, \ \ i\geq 1.
\end{equation}
Similarly, since $\sum_{k=m}^{\infty}\frac{H_m^{(k)}c(k)}{P(k,k-1)}$ for $k \geq m$ is nonnegative and finite, we have
\begin{equation}\label{finite-3}
\sum_{m=i+1}^{\infty}\sum_{k=m}^{\infty}\frac{H_m^{(k)}c(k)}{P(k,k-1)}<\sum_{m=1}^{\infty}\sum_{k=m}^{\infty}\frac{H_m^{(k)}c(k)}{P(k,k-1)}<\infty,\ \ i\geq 1.
\end{equation}
Combining Eqs.(\ref{finite-1})--(\ref{finite-3}), we conclude that for any fixed $i \in \mathbb{S}$, $\phi(i) < \infty$ if $\phi(0) < \infty$. Thus, the proof is complete.
\end{proof}

By Corollary \ref{col-G} and Theorem \ref{theorem-2}, we can derive the Green matrix of  downward skip-free Markov chains.

\begin{corollary}\label{G-downward-process}
Assume that $\boldsymbol{X}$ is an irreducible downward skip-free Markov chain.
Then, its Green matrix $\boldsymbol{G}$ is given by
\begin{equation*}
G(i,j)=\left\{\aligned & M(i), &   j=0,\\
  &M(i)\sum_{k=1}^{j}\frac{P_0^{(k+)}H_k^{(j)}}{P(j,j-1)}, &  j\leq i,\\
  & M(i)\sum_{k=1}^{j}\frac{P_0^{(k+)}H_k^{(j)}}{P(j,j-1)}-\sum_{k=i+1}^{j}\frac{H_k^{(j)}}{P(j,j-1)}, & j>i. \endaligned \right.
\end{equation*}
\end{corollary}

\begin{remark} \label{remark-transient-downward}
From Corollary \ref{G-downward-process}, we know that the downward skip-free Markov chain  is transient if and only if $M(0)<\infty$, i.e.,
\[\lim_{n\rightarrow\infty}\frac{\sum_{m=1}^{n}H_m^{(n)}}{\sum_{i=1}^{n} P_{0}^{(i+)}H_i^{(n)}}<\infty.\]
The results on recurrence of downward skip-free Markov chains can be found in  \cite{MYZ22}.
\end{remark}

\section{Examples}\label{sec3}

In this section, we apply our results to $GI/M/1$ queues and $M/G/1$ queues.

\begin{example}\label{exp1}
Let $\boldsymbol X$ be an embedded Markov chain of the $GI/M/1$ queue and its transition probability matrix $\boldsymbol{P}$ is given by
\begin{equation*}\label{P-GI/M/1}
 \boldsymbol{P}=\left (
 \begin{array}{ccccccc}
   b_{0} & a_{0} & 0 & 0&\cdots \\
   b_{1} & a_{1} & a_{0} & 0 & \cdots \\
   b_{2} & a_{2} & a_{1} &a_{0} & \cdots\\
   b_{3} & a_{3} & a_{2} &a_{1} &\cdots  \\
   \vdots &\vdots &\vdots &\vdots &\ddots \\
   \end{array}\right ),
\end{equation*}
where $a_k>0$, and $b_k=1-\sum_{i=0}^{k}a_i$, $k\geq 0$.
\end{example}

It is clear that the embedded Markov chain of the $GI/M/1$ queue is irreducible and follows an upward skip-free Markov chain structure.
In this model, we take
$$a_k=\frac{z-1}{z^{k+1}} \quad \text{and} \quad  b_k=\frac{1}{z^{k+1}},\ \ k\geq 0,$$
where the constant $z>1$.
By calculations, we obtain
\begin{equation*}\label{GI-M-1-1}
P_n^{(k-)}=\frac{1}{z^{n-k+1}},\ \ 0\leq k < n,
\end{equation*}
and
\begin{equation}\label{GI-M-1-2}
F_n^{(i)}=\frac{1}{z(z-1)^{n-i}}, \ \ 0\leq i< n.
\end{equation}
Furthermore, according to Eq.(\ref{GI-M-1-2}), we have
\[\sum_{n=0}^{\infty}F_n^{(0)}=1+\frac{1}{z}\sum_{n=1}^{\infty}\frac{1}{(z-1)^n}.\]

It follows from Remark \ref{remark-transient-upward}  that this chain is transient when $z>2$, and recurrent when $1<z\leq 2$.
Here, we consider the potential $\boldsymbol\phi$ in the case of transient embedded Markov chain of the $GI/M/1$ queue.
Now, let $c(i)=z^{-i}$, $i\in \mathbb{S}$, where $z>2$.
From Eq.(\ref{GI-M-1-2}) and the function $\boldsymbol{c}$, we have

\begin{equation}\label{GI-M-1-3}
 \sum_{k=0}^{m}\frac{F_{m}^{(k)}c(k)}{P(k,k+1)} =\frac{z}{(z-1)^{m+1}}, \ \ m\geq0.
\end{equation}
Finally, according to Theorem \ref{theorem-1} and Eq.(\ref{GI-M-1-3}), we have
\[\phi(i)=\sum_{m=i}^{\infty}\sum_{k=0}^{m}\frac{F_{m}^{(k)}c(k)}{q_{k,k+1}}=\frac{z}{(z-2)(z-1)^i}, \ \ i\in\mathbb{S}.\]
\begin{figure}[h]
  \centering
  \includegraphics[width=13cm]{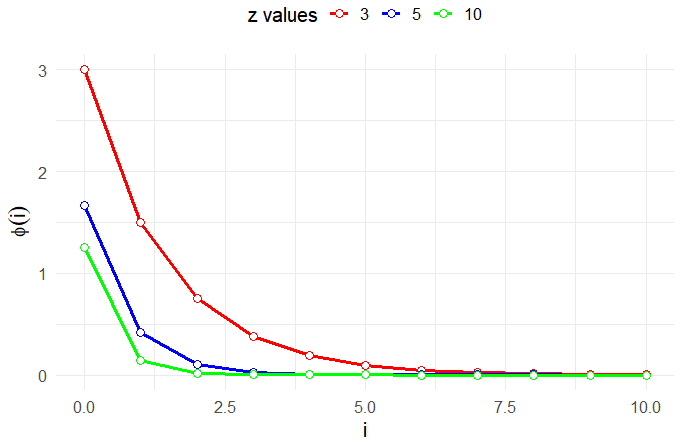}
  \caption{The value of the potential $\boldsymbol{\phi}$ with $z=3, 5, 10$ of Example \ref{exp1}. }
\end{figure}

\begin{example}\label{exp2}
Let $\boldsymbol X$ be an embedded Markov chain of the $M/G/1$ queue and its transition probability matrix $\boldsymbol{P}$ is given by
\begin{equation*}
\boldsymbol{P}=\left (
 \begin{array}{ccccccc}
   a_{0} & a_{1} & a_{2} & a_{3}&\cdots \\
   a_{0} & a_{1} & a_{2} & a_{3} & \cdots \\
   0 & a_{0} & a_{1} &a_{2} & \cdots\\
   0 & 0 & a_{1} &a_{2} &\cdots  \\
   \vdots &\vdots &\vdots &\vdots &\ddots \\
   \end{array}\right ),
\end{equation*}
where $a_k>0$, $k\geq 0$, and $\sum_{k=0}^{\infty}a_k=1$.
\end{example}

Obviously, the embedded Markov chain of the $M/G/1$ queue is an irreducible downward skip-free Markov chain.
Similar to Example \ref{exp1}, we take
\[a_k=\frac{z-1}{z^{k+1}},\ \ k\geq0,\]
where $z>1$. We first determine the recurrence of the chain.
By calculations, we obtain
\begin{equation}\label{mg1-1}
P_{n}^{(k+)}=\frac{1}{z^{k-n+1}},\ \ 1\leq n< k;\ \ \ P_{0}^{(k+)}=\frac{1}{z^{k}},\ \ k\geq1,
\end{equation}
and
\begin{equation}\label{mg1-2}
H_{n}^{(i)}=\frac{1}{z(z-1)^{i-n}},\ \ 1\leq n< i.
\end{equation}
Additionally, based on Eqs.(\ref{mg1-1}) and (\ref{mg1-2}), we obtain
\begin{equation}\label{mg1-3}
\sum_{k=1}^{n} P_{0}^{(k+)}H_k^{(n)}=\frac{1}{z(z-1)^{n-1}}, \ \ k\leq n,
\end{equation}
and
\begin{equation}\label{mg1-4}
\sum_{k=i+1}^{n}H_k^{(n)}=\frac{(z-1)^{n-i+1}-1}{z(z-2)(z-1)^{n-i-1}},\ \ n\geq i+1.
\end{equation}
If $1<z<2$, combining with Eqs.(\ref{mg1-3}) and (\ref{mg1-4}), we have
\begin{equation}\label{mg1-5}
M(i)=\lim_{n\rightarrow\infty}\frac{\sum_{k=i+1}^{n}H_k^{(n)}}{\sum_{k=1}^{n} P_{0}^{(k+)}H_k^{(n)}}=\lim_{n\rightarrow\infty}{\frac{(z-1)^{n-i+1}-1}{(z-2)(z-1)^{-i}}}={\frac{(z-1)^{i}}{2-z}},\ \ i\in\mathbb{S};
\end{equation}
if $z>2$,
\begin{equation}\label{mg1-6}
M(i)=\lim_{n\rightarrow\infty}{\frac{(z-1)^{n-i+1}-1}{(z-2)(z-1)^{-i}}}=+\infty,\ \ i\in\mathbb{S};
\end{equation}
For the case of $z=2$, Eqs.(\ref{mg1-3}) and (\ref{mg1-4}) becomes
\[\sum_{k=1}^{n} P_{0}^{(k+)}H_k^{(n)}=\frac{1}{2}, \ \ k\leq n.\]
and
\[\sum_{k=i+1}^{n}H_k^{(n)}=\frac{1}{2}(n-i+1),\ \ n\geq i+1,\]
Furthermore, we have
\begin{equation}\label{mg1-7}
M(i)=\lim_{n\rightarrow\infty}\frac{\frac{1}{2}(n-i+1)}{\frac{1}{2}}=+\infty ,\ \ i\in\mathbb{S}.
\end{equation}

Combining Remark \ref{remark-transient-downward} with Eqs.(\ref{mg1-5})--(\ref{mg1-7}),
it is established that the embedded Markov chain of the $M/G/1$ queue is transient when $1 < z < 2$ and recurrent when $z \geq 2$.
For the case where the chain is transient, we choose $c(0)=0$, and $c(i) = \left(\frac{z-1}{z}\right)^{i}$, $i \geq 1$.
By calculations, we have
\begin{equation}\label{M/G/1-0}
\sum_{k=m}^{\infty}\frac{H_m^{(k)}c(k)}{P(k,k-1)}=\frac{(z-1)^{m-2}}{z^m}+\left(\frac{z-1}{z}\right)^{m-1}, \ \ m\geq1.
\end{equation}
Based on Eq.(\ref{M/G/1-0}), we have
\begin{equation}\label{M/G/1-1}
\delta=\sum_{m=1}^{\infty} P_{0}^{(m+)}\sum_{k=m}^{\infty}\frac{H_m^{(k)}c(k)}{P_{k,k-1}}+c(0)=\frac{1}{z-1}
\end{equation}
and
\begin{equation}\label{M/G/1-2}
\sum_{m=i+1}^{\infty}\sum_{k=m}^{\infty}\frac{H_m^{(k)}c(k)}{P(k,k-1)}=\frac{(z^2-z+1)(z-1)^{i-1}}{z^i}, \ \  i \in\mathbb{S}.
\end{equation}
Finally, combining Theorem \ref{theorem-2} with Eqs.(\ref{mg1-5}) and (\ref{M/G/1-1})--(\ref{M/G/1-2}), we obtain that
 \begin{equation*}
  \phi(i)=\left(\frac{1}{2-z}-\frac{z^2-z+1}{z^i}\right)(z-1)^{i-1}, \ \ i \in\mathbb{S}.
 \end{equation*}
\begin{figure}[h]
  \centering
  \includegraphics[width=13cm]{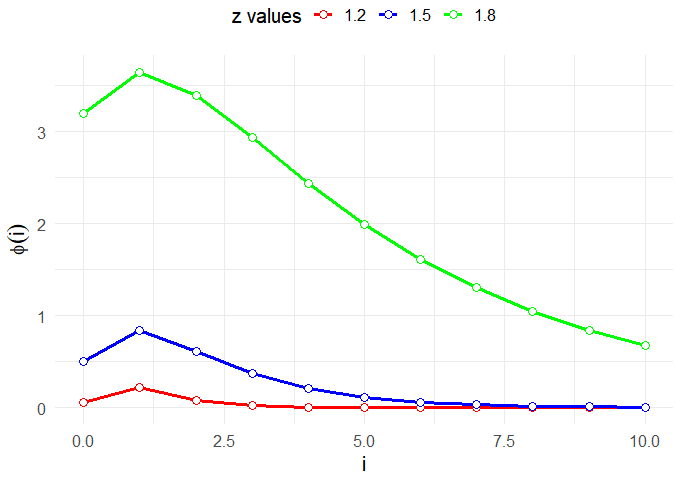}
  \caption{The value of the potential $\boldsymbol{\phi}$ with $z=1.2, 1.5, 1.8$ of Example \ref{exp2}. }\label{fig2}
\end{figure}

\begin{remark}
According to Theorem \ref{theorem-1}, when the potential of upward skip-free Markov chains is finite, it must be monotonically decreasing.
However, this is not necessarily the case for downward skip-free Markov chains, as demonstrated in Figure \ref{fig2}.
\end{remark}

\section{Continuous-time skip-free Markov chains}\label{sec4}

In this section, we discuss the potentials of continuous-time skip-free Markov chains.
We initially introduce the potential theory for general continuous-time Markov chains (CTMCs).
Let $\boldsymbol{X}=\{X_t, t\geq0\}$ be an irreducible CTMC on the state space ${\mathbb{S}}$, with a totally stable and regular $Q$-matrix $\boldsymbol{Q}=(Q(i,j))_{i,j\in{\mathbb{S}}}$.
Given a CTMC $\boldsymbol{X}$ and a finite nonnegative function $\boldsymbol{c}$, we define the function $\boldsymbol{\psi}$ as the following expectation of an integrable functional:
\[\psi(i)=\mathbb{E}_{i}\bigg[\int_{0}^{\infty}c(X_t)\mathrm{d}t\bigg],\ \ i\in{\mathbb{S}}.\]
Then, $\boldsymbol{\psi}$ represents the potential of the chain $\boldsymbol{X}$ and the function $\boldsymbol{c}$, and it is the minimal nonnegative solution of the following equation:
\begin{equation}\label{ctmc-1}
-\boldsymbol{Q}\boldsymbol{\psi}=\boldsymbol{c}.
\end{equation}
For more details, see, e.g., \cite{HG88}.

Now, let $_{(n)}\boldsymbol{Q}$ be the $(n+1)\times (n+1)$ northwest corner of $\boldsymbol{Q}$,
and let ${_{(n)}\boldsymbol{X}} = \{{_{(n)}{X}_t}, t \geq 0\}$ be the CTMC with $Q$-matrix $_{(n)}\boldsymbol{Q}$. The potential of  ${_{(n)}\boldsymbol{X}}$ and $_{(n)}\boldsymbol{c}$ is given by
\[_{(n)}\psi(i)=\mathbb{E}_{i}\bigg[\int_{0}^{\infty}{_{(n)}c}({_{(n)}X_t})\mathrm{d}t\bigg],\ \ i\in{_{(n)}\mathbb{S}}.\]

Using a technique similar to the one in Theorem \ref{theorem-general-mc}, we can derive the following convergence theorem for the potential $\boldsymbol{\psi}$ of CTMCs.

\begin{theorem}\label{theorem-general-ctmc}
Assume that the CTMC $\boldsymbol{X}$ is irreducible with a totally stable and regular $Q$-matrix $\boldsymbol{Q}$,
and the nonnegative function $\boldsymbol{c}$ on $\mathbb{S}$ is finite. Then, we have
\[\lim_{n\rightarrow\infty}{_{(n)}}\psi(i)\uparrow\psi(i),\ \ i\in\mathbb{S}.\]
\end{theorem}
\begin{proof}
According to Eq.(\ref{ctmc-1}), we know that ${_{(n)}}\boldsymbol{\psi}$ is the minimal nonnegative solution of the following equation:
\begin{equation}\label{ctmc-2}
-{_{(n)}\boldsymbol{Q}}{{_{(n)}}\boldsymbol{\psi}}={_{(n)}\boldsymbol{c}}.
\end{equation}
On the other hand, Eq.(\ref{ctmc-1})  can be rewritten as
\begin{equation*}
\boldsymbol{\psi}=\boldsymbol{P_{Q}}\boldsymbol{\psi}+\boldsymbol{c_Q},
\end{equation*}
where $\boldsymbol{P_{Q}}$ is the transition probability matrix of the embedded chain of the CTMC, i.e.,
\begin{equation*}
 {\boldsymbol{P_{Q}}}=\left (
 \begin{array}{ccccccc}
   0 & -\frac{Q(0,1)}{Q(0,0)} & -\frac{Q(0,2)}{Q(0,0)} &-\frac{Q(0,3)}{Q(0,0)} &\cdots \\
  -\frac{Q(1,0)}{Q(1,1)}   &0 &   -\frac{Q(1,2)}{Q(1,1)}  &  -\frac{Q(1,3)}{Q(1,1)}  & \cdots  \\
 -\frac{Q(2,0)}{Q(2,2)}  &  -\frac{Q(2,1)}{Q(2,2)}  &0 &  -\frac{Q(2,3)}{Q(2,2)} & \cdots \\
   -\frac{Q(3,0)}{Q(3,3)} & -\frac{Q(3,1)}{Q(3,3)} &-\frac{Q(3,2)}{Q(3,3)}& 0 &\cdots  \\
       \vdots &\vdots  &\vdots &\vdots&\ddots \\
   \end{array}\right ),
\end{equation*}
and
\[\boldsymbol{c_Q}=\bigg(-\frac{c(0)}{Q(0,0)},-\frac{c(1)}{Q(1,1)},-\frac{c(2)}{Q(2,2)},\cdots\bigg).\]
Similarly, Eq.(\ref{ctmc-2}) can be rewritten as
\begin{equation*}
{_{(n)}\boldsymbol{\psi}}={_{(n)}\boldsymbol{P_{Q}}}{_{(n)}\boldsymbol{\psi}}+{_{(n)}\boldsymbol{c_Q}},
\end{equation*}
where \({_{(n)}\boldsymbol{P_{Q}}}\) and \({_{(n)}\boldsymbol{c_Q}}\) are the truncated \((n+1)\)-dimensional matrix and \((n+1)\)-dimensional function of \(\boldsymbol{P_{Q}}\) and \(\boldsymbol{c_Q}\), respectively.

The remaining part of the proof follows the same approach as in Theorem \ref{theorem-general-mc}, and is therefore omitted here.
\end{proof}

Next, we give the potentials $\boldsymbol{\psi}$ of continuous-time upward and downward skip-free Markov chains.
We call that the CTMC $\boldsymbol{X}$ is a continuous-time upward skip-free Markov chain
if its $Q$-matrix $\boldsymbol{Q}$ satisfies $Q(i,i+1)>0$ for all $i\geq0$ and $Q(i,k)=0$ for all $k-i\geq2$; and
$\boldsymbol{X}$ is a continuous-time downward skip-free Markov chain if its $Q$-matrix $\boldsymbol{Q}$ satisfies
 $Q(i,i-1)>0$ for all $i\geq1$ and $Q(i,k)=0$ for all $i-k\geq2$.
 Similar to Section \ref{sec2}, we define the following notations:
\begin{equation*}\label{Notation-1-1}
{Q_{n}^{(k-)}}=\sum_{i=0}^{k}Q(n,i),\ \ 0\leq k<n,
\end{equation*}
\begin{equation*}\label{Notation-2-1}
\widetilde{F}_{i}^{(i)}=1, \ \ i\geq 0,\ \ \widetilde{F}_{n}^{(i)}=\frac{1}{Q(n,n+1)}\sum_{k=i}^{n-1}{Q_{n}^{(k-)}}\widetilde{F}_{k}^{(i)}, \ \ 0\leq i<n;
\end{equation*}
and
\[
Q_{m}^{(k+)}=\sum_{i=k}^{\infty}Q(m,i),\ \ k>m\geq0,
\]
\[
\widetilde{H}_{i}^{(i)}=1, \ \ \widetilde{H}_{m}^{(i)}=\frac{1}{Q(m,m-1)}\sum_{k=m+1}^{i}Q_{m}^{(k+)}\widetilde{H}_{k}^{(i)}, \ 1\leq m<i.
\]

According to Theorem \ref{theorem-general-ctmc} and the same approach as in Theorems \ref{theorem-1} and \ref{theorem-2} ,
we have the following results.

\begin{theorem}
Assume that $\boldsymbol{X}$ is an irreducible continuous-time upward skip-free Markov chain with a totally stable and regular $Q$-matrix $\boldsymbol{Q}$,
and the nonnegative function $\boldsymbol{c}$ on $\mathbb{S}$ is finite.
Then, we have
 \begin{equation}\label{theorem-5}
  \psi(i)=\sum_{m=i}^{\infty}\sum_{k=0}^{m}\frac{\widetilde{F}_{m}^{(k)}c(k)}{Q(k,k+1)}, \ \ i \in\mathbb{S},
 \end{equation}
and this is finite if and only if $\psi(0)<\infty$.
\end{theorem}

\begin{theorem}
Assume that $\boldsymbol{X}$ is an irreducible continuous-time downward skip-free Markov chain with a totally stable and regular $Q$-matrix $\boldsymbol{Q}$,
and the nonnegative function $\boldsymbol{c}$ on $\mathbb{S}$ is finite.
Then, we have
 \begin{equation}\label{theorem-6}
  \psi(i)={\eta} \widetilde{M}(i)-\sum_{m=i+1}^{\infty}\sum_{k=m}^{\infty}\frac{\widetilde{H}_m^{(k)}c(k)}{Q(k,k-1)}, \ \ i \in\mathbb{S},
 \end{equation}
 where
 \[\eta=\sum_{i=1}^{\infty} Q_{0}^{(i+)}\sum_{k=i}^{\infty}\frac{\widetilde{H}_i^{(k)}c(k)}{Q(k,k-1)}+c(0),\ \ \widetilde{M}(i)=\lim_{n\rightarrow\infty}\frac{\sum_{m=i+1}^{n}\widetilde{H}_m^{(n)}}{\sum_{i=1}^{n} Q_{0}^{(i+)}\widetilde{H}_i^{(n)}},\]
and this is finite if and only if $\psi(0)<\infty$.
\end{theorem}

\begin{remark}
When the CTMC $\boldsymbol{X}$ is a continuous-time birth-death Markov chains (birth-death process), then
(\ref{theorem-5}) or (\ref{theorem-6}) becomes
\[\psi(i)=\sum_{m=i}^{\infty} \frac{1}{Q(m,m+1)\pi(m)}\sum_{k=0}^{m}\pi(k)c(k), \ \ i\in\mathbb{S},\]
where $\pi(0)=1$, and $\pi(k)=\prod_{m=1}^{k}\frac{Q(m-1,m)}{Q(m,m-1)}$ for $k\geq 1$. The results can also be referred to \cite{P03}.
\end{remark}

\end{document}